\documentclass[twoside, reqno, 10pt]{amsart}

\usepackage[left=4cm, right=4cm, top=4cm, bottom=4cm]{geometry}

\usepackage[colorlinks=true, pdfstartview=FitV, linkcolor=blue,citecolor=blue, urlcolor=blue]{hyperref}

\usepackage[usenames]{xcolor} 

\definecolor{labelkey}{rgb}{0,0,1}
\definecolor{Red}{rgb}{0.7,0,0.1}
\definecolor{Green}{rgb}{0,0.7,0}

\usepackage{amsfonts, amssymb, amsmath, amsthm, mathrsfs, bbm, cjhebrew, gensymb, textcomp, mathtools, dsfont, calligra}

\usepackage[normalem]{ulem}

\usepackage{commath, setspace, subcaption, parcolumns, multirow, multicol, accents, comment, marginnote, verbatim, empheq, enumerate, stackrel, enumitem, float, physics}
\usepackage[capitalize,nameinlink,noabbrev]{cleveref}

\usepackage{graphicx, graphics, epsfig, psfrag, tikz, tikz-cd,svg}

\usepackage{relsize}
\usepackage{scalerel}
\usepackage{dsfont}
\usepackage[sort,nocompress]{cite}
\usepackage[toc,page]{appendix}
\usepackage{bbm}

\numberwithin{equation}{section}

\usepackage{tikz}

\newtheorem{Thm}{Theorem}[section]
\newtheorem{Lem}[Thm]{Lemma}

\newtheorem{Rmk}[Thm]{Remark}

\newtheorem*{Thm*}{Theorem}

\newcommand{\vect}[1]{\mathbf{#1}}

\newcommand{\bu}{\vect{u}}
\newcommand{\bv}{\vect{v}}

\newcommand{\bx}{\vect{x}}

\newcommand{\bbu}{\bar{\mathbf{u}}}
\newcommand{\tbu}{\widetilde{\mathbf{u}}}
\newcommand{\bw}{\vect{w}}

\newcommand{\bbf}{\vect{f}}

\newcommand{\NN}{\mathbb{N}}
\newcommand{\ZZ}{\mathbb{Z}}

\newcommand{\RR}{\mathbb{R}}

\newcommand{\no}[2]{\lVert#2\rVert_{#1}}

\newcommand{\goesto}{\rightarrow}
\newcommand{\smod}{\setminus}

\newcommand{\al}{\alpha}
\newcommand{\be}{\beta}
\newcommand{\de}{\delta}
\newcommand{\De}{\Delta}
\newcommand{\gam}{\gamma}

\newcommand{\eps}{\epsilon}

\newcommand{\lam}{\lambda}

\newcommand{\kap}{\kappa}

\newcommand{\tht}{\theta}
\newcommand{\om}{\omega}
\newcommand{\Om}{\Omega}
\newcommand{\ph}{\varphi}

\newcommand{\bdy}{\partial}
\newcommand{\lb}{\langle}
\newcommand{\rb}{\rangle}

\newcommand{\rone}[1]{#1}
\newcommand{\rtwo}[1]{#1}
\newcommand{\add}[1]{#1}

\usepackage[backgroundcolor=gray!30,linecolor=black]{todonotes}

\usepackage{multicol}
\usepackage[normalem]{ulem}

 \title[Remarks on the stabilization of large scales in 2D KSE]{Remarks on the stabilization of large-scale growth in the 2D {\rone{Kuramoto--Sivashinsky}} equation}

 \author{Adam Larios and Vincent R. Martinez$^\dagger$}
\begin{document}

\begin{abstract}
In this article, \add{some} elementary observation\add{s} \add{are} made regarding the behavior of solutions to the two-dimensional curl-free Burgers equation which suggest the distinguished role played by the scalar divergence field in determining the dynamics of the solution. These observations inspire a new divergence-based regularity condition for the two-dimensional {\rone{Kuramoto--Sivashinsky}} equation (KSE) that provides conceptual clarity to the nature of the potential blow-up mechanism for this system. The relation of this regularity criterion to the {\rone{{\rone{Ladyzhenskaya--Prodi--Serrin}}}}-type criterion for the KSE is also established, thus providing the basis for the development of an alternative framework of regularity criterion for this equation based solely on the low-mode behavior of its solutions. The article concludes by applying these ideas to identify a conceptually simple modification of KSE that yields globally regular solutions, as well as {\rone{providing}} a straightforward verification of this regularity criterion to establish global regularity of solutions to the 2D {\rone{Burgers--Sivashinsky}} equation. The proofs are direct, elementary, and concise.
\end{abstract}

\maketitle

\vspace{1em}

{\noindent \small {\it {\bf Keywords:} 2D {\rone{Kuramoto--Sivashinsky}} equation, curl-free Burgers equation, global regularity, regularity criterion, low-mode regularity criterion, {\rone{{\rone{Ladyzhenskaya--Prodi--Serrin}}}}}
  }

\setcounter{tocdepth}{1}

\section{Introduction}\label{sect:intro}

The Kuramoto-Sivashinky equation was originally introduced in \cite{Sivashinsky_1980_stoichiometry} as a model for the propagation of flame fronts and identified in connection with reaction-diffusion systems in \cite{Kuramoto_1978,Kuramoto_Tsuzuki_1975,Kuramoto_Tsuzuki_1976}, motivated by the stability of traveling waves. Numerical studies such as \cite{Michelson_Sivashinsky_1977_numerical} indicated chaotic dynamics of the system whose complexity increases as \add{the length of the domain increases}. When $d=1$, it was proved in \cite{Nicolaenko_Scheurer_1984} (see also \cite{Tadmor_1986}) that the initial value problem \eqref{eq:KSE:scalar} is globally well-posed in $H^2(\Om)$. Soon after, it was then shown in \cite{Nicolaenko_Scheurer_Temam_1985} that \eqref{eq:KSE:scalar} possessed finitely many determining modes and a finite-dimensional global attractor. In the celebrated work \cite{Foias_Nicolaenko_Sell_Temam_1988}, it was furthermore shown that the long-time behavior of \eqref{eq:KSE:scalar} was completely characterized by a finite and low-dimensional system. Since these seminal works, the KSE has become an important test bed \add{for} probing the connections between PDEs, dynamical systems, chaotic behavior, and turbulence \cite{Hyman_Nicolaenko_1986, Hyman_Nicolaenko_Zaleski_1986}, finding effective computational means for studying the dynamics of infinite-dimensional systems exhibiting such a strong form of finite-dimensional behavior (see, e.g., \cite{Jolly_Kevrekidis_Titi_1991, Foias_Jolly_Kevrekidis_Titi_1994_KSE, Johnson_Jolly_Kevrekidis_1997, Jolly_Rosa_Temam_2000, Johnson_Jolly_Kevrekidis_2001, Dieci_Jolly_Rosa_VanVleck_2008, Bezia_Mabrouk_2019} and the references therein), and \add{in some recent works} for the testing of model discovery, techniques in data assimilation,  and parameter estimation techniques \cite{Larios_Pei_2017_KSE_DA_NL,Lunasin_Titi_2015,Pachev_Whitehead_McQuarrie_2021concurrent,mojgani2022discovery}.

The (non-dimensionalized) Kuramoto-Sivashinky equation is given by
    \begin{align}\label{eq:KSE:scalar}
        \bdy_t\phi+{\add{\frac{1}2}}|\nabla\phi|^2=-\De^2\phi-\lam\De\phi,
    \end{align}
where $\phi:\RR^2\goesto\RR^2$, $\lam>0$. Upon applying the gradient operator to \eqref{eq:KSE:scalar}, one obtains the corresponding vector-form of \eqref{eq:KSE:scalar}:
    \begin{align}\label{eq:KSE}
        \bdy_t\bu+\bu\cdotp\nabla\bu=-\De^2\bu-\lam\De\bu,\quad \nabla^\perp\cdotp\bu=0,
    \end{align}
{\rone{where $\bu:\RR^2\goesto\RR^2$, $\nabla^\perp=(-\bdy_2,\bdy_1)$, $\De=\bdy_1^2+\bdy_2^2$, and $\De^2=(\bdy_1^2+\bdy_2^2)^2$}}. We note that \eqref{eq:KSE:scalar}, \eqref{eq:KSE} represent a particular choice of re-scaling of the original system, so that the parameters characterizing the domain size are ultimately encoded in $\lam$. Indeed, if $L$ is the side-length of a square domain, then under a particular choice of re-scaling, where $L$ is order $1$, one has $\lam\sim L^2$. In this article, we are interested in the issue of global regularity for \eqref{eq:KSE} over the periodic box $\Om=[0,2\pi]^d$, where $d\geq1$.

In spite of the developments mentioned above, two important issues remain unresolved. In the case $d=1$, numerical evidence \cite{Fantuzzi_Wynn_2015, Goluskin_Fantuzzi_2019,Pomeau_Pumir_Pelce_1984, Wittenberg_Holmes_1999} shows that the average energy, ${\rone{\limsup_{t\goesto\infty}\frac{1}t\int_0^t\|\bu(t)\|_{L^2}^2dt}}$, is an \textit{intensive quantity}, namely, that it obtains a bound independent of $\lam$ as $\lam\goesto\infty$, consistent with the viewpoint of thermodynamic perspective of viewing \eqref{eq:KSE} as a ``large'' system as $\lam\goesto\infty$ with spatially localized interacting subsystems that allow for short-time decorrelated interactions. This has yet to be confirmed rigorously, although some progress has been made \cite{Bronski_Gambill_2006, Collet_Eckmann_Epstein_Stubbe_1993_Analyticity, Giacomelli_Otto_2005_KSE, Goldman_Josien_Otto_2015, Goodman_1994,Nicolaenko_Scheurer_Temam_1985, Otto_2009} with the best-known uniform-in-time bounds on the energy obtained as $o(\lam)$ \cite{Giacomelli_Otto_2005_KSE} and best-known time-average bounds obtained as $O(\lam^{(1/3)^+})$
\cite{Goldman_Josien_Otto_2015, Otto_2009}. On the other hand, in higher-dimensions $d\geq2$, the issue of global existence of strong solutions for arbitrary large, finite energy initial data is still not known. It is this latter issue that the present article is concerned with. 

Local-in-time existence of analytic solutions was established in \cite{Biswas_Swanson_2007_KSE_Rn}, while the analog of the {\rone{Ladyzhenskaya--Prodi--Serrin}} regularity criterion, originally developed for the 3D {\rone{Navier--Stokes}} equations, was established for the $d-$dimensional KSE in \cite{Larios_Rahman_Yamazaki_2021_JNLS_KSE_PS}. Several  works \cite{Ambrose_Mazzucato_2018, Ambrose_Mazzucato_2021, Benachour_Kukavica_Rusin_Ziane_2014_JDDE_2DKSE, Kukavica_Massatt_2023} have constructed global solutions by exploiting anisotropy in some way.  Another approach to understanding the issue of global regularity has been to identify various mechanisms such as maximum principles \cite{Ibdah_2021_MichelsonSivashinsky, Larios_Yamazaki_2020_rKSE, Massatt_2022_DCDSB}, mixing \cite{CotiZelati_Dolce_Feng_Mazzucato_2021}, dispersion \cite{Bartuccelli_Deane_Gentile_2020_JDDE}, ``algebraic calming'' \cite{Enlow_Larios_Wu_2023_calming} that ultimately stabilize large-scale growth in the system and then to either modify \eqref{eq:KSE:scalar} to possess these mechanisms or to locate them in closely related systems which have these mechanisms naturally.

A major aim of the present work is to highlight the role that the sign of the divergence of the solution plays in the (possible) destabilization of the large-scales in the system.  In \cite{Larios_Rahman_Yamazaki_2021_JNLS_KSE_PS}, it was shown (among other results) that bounding various norms of the divergence is sufficient to prove that the $d$-dimensional KSE is globally well-posed, although this criterion was probably observed at least informally by other researchers earlier.  In the present work, we prove that it is sufficient to control merely the positive part of the divergence, which seems not to have been observed earlier, although some recent numerical experiments in \cite{Larios_Rahman_Yamazaki_2021_JNLS_KSE_PS} had hinted at this.  The reason for this can be understood dynamically as follows: In the KSE, small scales are stabilized by the hyperdiffusive mechanism, while large scales (smaller than order $\sqrt{\lambda}$) are destabilized by the backward diffusive mechanism\footnote{It was first observed by E.S. Titi \cite{EdrissTitiPrivateCommunication} that in 1D at least, the equation is stabilized by the nonlinear cascading energy from large scales to small scales.}.  In higher dimensions ($d\geq2$), we postulate the following destabilizing mechanism: Regions of positive divergence tend to expand (since the vector field ``points outward'' {\add{in these regions}}).  Since the direction of the {\add{local}} expansion is in the same direction as the solution $\bu$, the nonlinear term $\bu\cdot\nabla\bu$ maintains this expansion, at least for short times, by advecting the velocity field along this expansion direction, further increasing the divergence in these regions.  This feedback mechanism creates large bubble-like regions of positive divergence\footnote{These regions were observed and explicitly plotted in computational simulations in \cite{Larios_Rahman_Yamazaki_2021_JNLS_KSE_PS}, but hints of this can be seen earlier in \cite{Kalogirou_Keaveny_Papageorgiou_2015}, and even as far back as \cite{Sivashinsky_1980_stoichiometry}.  Note that such a dominance of positive divergence regions over negative divergence regions is \textit{not} observed in 1D simulations; this can be explained by the nonlinear term vanishing in $L^2$ energy estimates, which occurs only in dimension $1$.}.  These regions are naturally associated with large wave-lengths, and hence the low-mode instability of the equation causes these regions to grow rapidly. Since the average of the divergence over a periodic domain must be zero, there will always be some regions of negative divergence, but these regions will be thin\footnote{Regions of negative divergence must either be widespread (e.g., many thin filaments) or have large amplitude of divergence to compensate; computational evidence in \cite{Larios_Rahman_Yamazaki_2021_JNLS_KSE_PS} (see Fig. 7 in that work) indicate it is the latter.} as regions of positive divergence compete for dominance, and hence small-scales are also activated via this {\add{upwelling-type behavior}}.  Thus, there is some hope for global well-posedness via the ``Titi mechanism'' of stabilization via energy cascading to higher modes where it can be dissipated. However, this cascade needs to be strong enough to counter-balance the growth of positive-divergence regions.  We make these heuristic arguments more concrete in the present article by showing that if the divergence is bounded, or  \add{its} positive part \add{is bounded}, or \add{its behavior on} low-modes of the divergence is under control, {\add{in a precise sense that we specify below}}, then the system is globally well-posed.

{\add{We}} develop several new regularity criteria (\cref{thm:div}, \cref{thm:div:low}, \cref{thm:N}) that improve some of the criteria developed in \cite{Larios_Rahman_Yamazaki_2021_JNLS_KSE_PS}. The main motivation for each of these regularity criteria is based on a simple observation regarding solutions of the 2D curl-free Burgers equation (\cref{sect:Burgers}) in conjunction with the prevailing belief that one need only control low-mode instabilities to prevent blow-up. We introduce an elementary regularity criterion consistent with the observations regarding the role of the divergence of the vector field. This criterion is subsequently refined by making use of a beautiful idea introduced in \cite{Cheskidov_Shvydkoy_2014}. There, a unified approach to regularity of the 3D {\rone{Navier--Stokes}} equations was developed based on Kolmogorov's phenomenological theory of turbulence by exploiting the presence of a viscous cut-off in the energy spectrum; these ideas are applied in an analogous fashion in the context of KSE due to the presence of hyperviscosity, but to produce a decidedly different quantity from the context considered in \cite{Cheskidov_Shvydkoy_2014} for controlling regularity of solutions.  We then develop the ideas encapsulated by our regularity criterion to propose a different modification of the vector-formulation \eqref{eq:KSE} of KSE (\cref{sect:castrate}), as well as provide a simple demonstration of our regularity criteria in the particular case of the {\rone{Burgers--Sivashinsky}} model (\cref{sect:BSE}). Several insights are drawn from these results which are captured in various remarks, revealing several interesting future directions to pursue (\cref{rmk:EZKSE}, \cref{rmk:sharpness}, \cref{rmk:restrict}, \cref{rmk:control}). Finally, we summarize our paper and make some concluding remarks in \cref{sect:conclude}.

\section{Notation}

Given real numbers $a,b$, we define $a\wedge b=\min\{a,b\}$ and $a\vee b=\max\{a,b\}$. For $p\in[1,\infty)$, let $L^p=L_0^p(\Om)$ denote the space of $p$-integrable scalar fields which are $2\pi$-periodic in each direction, equipped with the norm
    \begin{align}\label{def:Lp}
        \no{L^p}{\ph}:=\left(\int_\Om|\ph(x)|^pdx\right)^{1/p},
    \end{align}
with the usual modification when $p=\infty$. We denote the $L^2$-based Sobolev space of order $k\geq0$, where $k$ is an integer, by $H^k=H^k(\Om)$, equipped with the norm
    \begin{align}\label{def:Hk}
        \no{H^k}{\ph}^2:=\no{L^2}{\ph}^2+\sum_{|\al|=k}\no{L^2}{\bdy^\al\ph}^2,
    \end{align}
where $\al\in(\mathbb{N}\cup\{0\})^2$ denotes a multi-index. Throughout the article, we will abuse notation by letting $L^p$, $H^k$ denote the same spaces but with vector-valued outputs. We will also make use of fractional homogeneous Sobolev spaces, $\dot{H}^\kap$, for $\kap\in\RR$. These spaces can be characterized in terms of the fractional Laplacian $D^\kap:=(-\De)^{\kap/2}$, which is defined via its Fourier transform:
    \begin{align}\label{def:Dkap}
        \widehat{D^\kap\ph}_\ell=|\ell|^{\kap}\hat{\ph}_\ell,
    \end{align}
where $\hat{\ph}_\ell=(2\pi)^{-2}\int_{\Om}\ph(x)e^{-i\ell\cdotp x}dx$, for each $\ell\in\ZZ^2$. The norm characterizing $\dot{H}^\kap$ is then defined by
    \begin{align}\label{def:Hkap}
        \no{\dot{H}^\kap}{\ph}:=\no{L^2}{D^\kap\ph}.
    \end{align}
Note that by the Plancherel identity,
    \begin{align}\label{eq:Hk:Hkap}
        \no{\dot{H}^\kap}{\ph}^2=(2\pi)^2\sum_{\ell\in\ZZ^2\smod\{\mathbf{0}\}}|\ell|^{2\kap}|\hat{\ph}_\ell|^2.
    \end{align}
We therefore define for $\kap\in\RR$:
	\begin{align}\label{def:Hkap:negative}
		\no{H^\kap}{\ph}:=\left(\no{L^2}{\ph}^2+\no{\dot{H}^\kap}{\ph}^2\right)^{1/2}.
	\end{align}
Observe that for $\kap=k$, the following quantities are equivalent (as norms):
	\begin{align}\label{eq:Hk:consistency}
		\no{H^k}{\ph}\sim\left(\no{L^2}{\ph}^2+\no{\dot{H}^\kap}{\ph}^2\right)^{1/2}.
	\end{align}
In particular, the definitions \eqref{def:Hk} and \eqref{def:Hkap} are consistent whenever $\kap$ is an integer. It is also readily verifiable that for all $\kap\in\RR$, we have the following equivalence:
	\begin{align}\label{def:Hkap:negative:equiv}
		\no{H^\kap}{\ph}\sim\left((2\pi)^2\sum_{\ell\in\ZZ^2}(1+|\ell|^2)^{\kap}|\hat{\ph}_\ell|^2\right)^{1/2}.
	\end{align}

Lastly, we observe that in the subspace of mean-free, $2\pi$-periodic functions (in each direction) over $\Om$, we may identify $H^\kap$ with $\dot{H}^\kap$. \textit{For the remainder of the manuscript, we will therefore adopt the abuse of notation that $H^\kap$ also denotes the homogeneous space whenever we are in the context of mean-free functions.}

We will make use of several inequalities. Let us recall the {\rone{Kato--Ponce}} inequality (see, e.g., \cite{Li_2019_Kato_Ponce,Grafakos_Oh_2014_Kato_Ponce} and the references therein): given $\kap\in(0,\infty)$, $r\in(1,\infty)$, and $1< p_j,q_j<\infty$ such that $1/p_j+1/q_j=1/r$, there exists an absolute constant $C_{KP}=C({\rone{\kappa}},r,p_1,q_1,p_2,q_2)>0$ such that
    \begin{align}\label{est:kato:ponce}
        \no{L^r}{D^\kap(\ph_1\ph_2)}\leq C_{KP}\left(\no{L^{p_1}}{D^\kap \ph_1}\no{L^{q_1}}{\ph_2}+\no{L^{p_2}}{\ph_1}\no{L^{q_2}}{D^\kap \ph_2}\right).
    \end{align}

If $f=(I-P_N)f$, for some $N>0$, we have the following Bernstein-type inequalities, which are straight-forward to prove: for real numbers $\kap<\kap'$, we have
    \begin{align}\label{est:bernstein}
    	\begin{split}
    	\no{H^{\kap'}}{P_N\ph}&\leq N^{\kap'-\kap}\no{L^2}{D^\kap \ph},\quad \kap\geq0
	\\
        \no{H^{\kap}}{(I-P_N)\ph}&\leq N^{-(\kap'-\kap)}\no{L^2}{D^{\kap'}\ph},\quad\kap'\geq0.
        \end{split}
    \end{align}
We point out that when $N\geq1$, then $(I-P_N)\ph$ is mean-free. We will also make use of the standard interpolation inequality in Sobolev spaces: given any real numbers $\kap_1<\kap<\kap_2$, we have
    \begin{align}\label{est:interpolation}
        \no{L^2}{D^\kap\ph}\leq \no{H^{\kap_2}}{\ph}^{\frac{\kap-\kap_1}{\kap_2-\kap_1}}\no{H^{\kap_1}}{\ph}^{\frac{\kap_2-\kap}{\kap_2-\kap_1}}.
    \end{align}

Lastly, we recall the Sobolev Embedding Theorem (see, e.g., \cite{Benyi_Oh_2013_Sobolev} for the fractional version on the torus): given $p\in(2,\infty)$ and real number $\kap\geq0$ such that $1/p=1/2-\kap/2$, there exists a constant $C_S=C(p,\kap)$ such that
    \begin{align}\label{est:sobolev}
        \no{L^p}{\ph}\leq C_S\no{H^\kap}{\ph}.
    \end{align}

\section{2D curl-free Burgers equation}\label{sect:Burgers}

It is well-known that the main obstacle to global well-posedness of the $d$-dimensional KSE ($d\geq2$) is in obtaining bounds on {\rone{the}} energy $\frac{1}2\no{L^2}{\bu}^2$. A similar equation in which the same obstacle is present is the $d$-dimensional hyper-dissipative Burgers equation 
    \begin{align}\label{eq:Burgers:hyper}
        \bdy_t\bu+\bu\cdotp\nabla\bu=-(-\De)^{\gam}\bu.
    \end{align}
where $\gam>1$ (this was first observed in \cite{Larios_Titi_2015_BC_Blowup} with $\gam=2$). When $\gam=1$, \eqref{eq:Burgers:hyper} is simply the viscous Burgers equation and global regularity of solutions is known due to the availability of an $L^\infty$-maximum principle (supposedly first proved in \cite{Ladyzhenskaya_1968}, but see the discussion and a modern proof in \cite{Pooley_Robinson_2016}). However, when $\gam>1$, a maximum principle is not known to exist for either \eqref{eq:Burgers:hyper} or \eqref{eq:KSE}. Needless to say, if such a property were available, then global regularity would follow.

One feature to recognize that is common to both \eqref{eq:KSE} and \eqref{eq:Burgers:hyper} is the underlying presence of the inviscid Burgers equation, i.e., \eqref{eq:Burgers:hyper} when $\lam=0$. We momentarily reflect on this presence and its implications in the issue of the global regularity of \textit{curl-free} solutions to \eqref{eq:KSE} and \eqref{eq:Burgers:hyper}. In particular, let us consider the 2D inviscid Burgers equation, that is,
    \begin{align}\label{eq:Burgers}
        \bdy_t\bu+\bu\cdotp\nabla\bu=0,
    \end{align}
It is easy to see that if $\bu$ is smooth and initially curl-free, then it remains curl-free. Indeed, suppose that $\bu$ satisfies \eqref{eq:Burgers}. We introduce the variables
    \begin{align}\label{def:div:curl}
        \de=\nabla\cdotp\bu,\quad\om=\nabla^\perp\cdotp\bu.   
    \end{align}
We also denote
    \begin{align}\label{def:D}
        D_t=\bdy_t+\bu\cdotp\nabla.
    \end{align}
Then a straightforward computation shows that
    \begin{align}\label{eq:Burgers:vort}
       D_t\om+\de\om=0.
    \end{align}
Hence $\om(t,X(t;a))=\exp\left(\add{-}\int_0^t\de(s,X(s,a))ds\right)\om(0,a)$, for all $a\in\Om$, where $X(t;a)$ denotes the Lagrangian particle position of $a$ at time $t$:
    \begin{align}\label{eq:particle}
        \frac{d}{dt}X(t;a)=\bu(t,X(t;a)),\quad X(0;a)=a.    
    \end{align}
Thus, if $\om(0,\cdotp)\equiv0$, then $\om(t,\cdotp)\equiv0$. For this reason, when \eqref{eq:Burgers} is initialized with a curl-free vector-field, we refer to \eqref{eq:Burgers} as the \textit{curl-free Burgers equation}. For the remainder of the section, we will thus assume that $\nabla^\perp\cdotp\bu=0$ holds.

On the other hand, upon applying the divergence operator to \eqref{eq:Burgers} and invoking the curl-free property\footnote{The Leibniz product rule directly gives $\nabla\cdot(\bu\cdot\nabla\bu) = (\nabla\bu) : (\nabla\bu)^T+ \bu\cdot\nabla\delta$, and the curl-free property implies $\nabla\bu = (\nabla\bu)^T$.}, one obtains
    \begin{align}\label{eq:Burgers:div}
        D_t\de=-|\nabla\bu|^2.
    \end{align}
This shows that $\de$ \textit{always decreases} along Lagrangian trajectories. In particular, if a particle initially possesses {\rone{negative divergence}}, then it will remain negative throughout its evolution. Although the additional linear terms in \eqref{eq:KSE} may, in general, counteract this phenomenon, we take this elementary observation as an underlying motivation for the divergence-based regularity criterion that we develop for \eqref{eq:KSE}.

\begin{Rmk}\label{rmk:EZKSE}
In contrast to the curl-free Burgers equation, the vorticity form of the three-dimensional Euler equations asserts the transport of vorticity along the fluid's velocity field and stretching of vorticity via the velocity's gradient:
	\begin{align}\label{eq:Euler}
		D_t\boldsymbol{\om}=\boldsymbol{\om}\cdot\nabla \bu
	\end{align}
As we will see below, the role of the divergence in the irrotational setting will play a role analogous to the vorticity in the incompressible setting from the point of view of regularity in the context of the 2D {\rone{Kuramoto--Sivashinsky}} equation \eqref{eq:KSE}. 
\end{Rmk}

\begin{Rmk}\label{rmk:Boritchev}
{\add{We also refer the reader to the paper of A. Boritchev \cite{Boritchev2014}, wherein a stochastic analog of the curl-free \textit{viscous} Burgers equation is analyzed from the perspective of hydrodynamic turbulence and bounds on averages of increments and energy spectra are obtained uniformly in the viscosity parameter.}}
\end{Rmk}

\section{2D {\rone{Kuramoto--Sivashinsky}} equation}\label{sect:KSE}

We consider \eqref{eq:KSE} and first study how the divergence of the vector field controls the growth solutions in any Sobolev norm. In fact, we show that the amplitude of outward divergence can affect growth of norms, while regions of inward divergence stabilizes, consistent with the observations from \cref{sect:Burgers}.

We will make use of the following local existence, regularity, and continuation result.

\begin{Thm}\label{thm:exist}
Let $\bu_0\in L^2$ {\add{such that $\nabla^\perp\cdotp\bu_0=0$ in the sense of distribution}}. Then there exists $T_0=T_0(\no{L^2}{\bu_0})>0$ and a unique solution $\bu\in C_w([0,T_0];L^2))\cap L^2(0,T_0;H^2)$ to \eqref{eq:KSE}. Moreover, for a.e. $t_0\in(0,T_0)$ and $k\geq1$, one has $\bu\in C([t_0,T_0];H^k)$ such that
    \begin{align}\notag
        \sup_{t\in[t_0,T_0]}\no{H^k}{\bu(t)}\leq C(t_0, T_0,\no{L^2}{\bu_0}).
    \end{align}
In particular, $u(t)\in C^\infty$, for all $t\in(0,T_0]$.
Lastly, if $T^*$ denotes the maximal time of existence and $\sup_{t\in[0, T^*)}\no{L^2}{\bu(t;\bu_0)}<\infty$, then $T^*=\infty$.
\end{Thm}

In an effort to make the presentation self-contained, the relevant details of the proof of \cref{thm:exist} are supplied in the appendix (see \cref{sect:appendix}), but we nevertheless refer the reader to \cite{Molinet_2000, Biswas_Swanson_2007_KSE_Rn} for additional details (see also \cite{Feng_Mazzucato_2020,Sell_Taboada_1992}). \add{Note that the details we provide in \cref{sect:appendix} are carried out without appealing to the curl-free condition. Thus \cref{thm:exist} holds for $\bu_0\in L^2$ as well, except the corresponding solution need not satisfy the curl-free condition.} Before we proceed, let us recall that if $\bu$ were mean-free, then the curl-free condition (via the Helmholtz decomposition) implies $\bu=\nabla\phi$, for some scalar potential field $\phi$. Thus, in a mean-free setting, the vector form \eqref{eq:KSE} is consistent with the scalar form \eqref{eq:KSE:scalar}; we refer the reader to \cref{sect:div:low} for additional remarks regarding the validity of the mean-free assumption.

\subsection{A divergence-based regularity criterion}\label{sect:div}

Before stating the main result of this section, we first introduce some additional notation.  Denote the positive and negative parts of $\de=\nabla\cdotp\bu$ by $\de_-=\max\{0,-\de\}$ and $\de_+=\max\{0,\de\}$ so that $\de=\de_+-\de_-$. Then denote
    \begin{align}\label{def:div:plus:max}
        \de_+^*=\de_+^*(t):=\sup_{x\in\Om}\de_+(t,x).
    \end{align}
Our main result is then stated as follows.

\begin{Thm}\label{thm:div}
Given $\bu_0\in L^2$, let $\bu$ denote the unique smooth solution of \eqref{eq:KSE} with initial data $\bu(0)=\bu_0$ over its maximal interval of existence $(0,T^*)$. If
    \begin{align}\label{cond:div}
        \int_0^{T^*}\de_+^*(s)\ ds<\infty,
    \end{align}
then $\sup_{t\in[0,T^*)}\no{L^2}{\bu(t)}<\infty$ and, subsequently, $T^*=\infty$. Conversely, if \eqref{cond:div} fails, then either $\limsup_{t\goesto T^*-}\no{H^3}{\bu(t)}$ or $\limsup_{t\goesto T^*-}\no{H^1}{\bu(t)}$ is infinite.
\end{Thm}

\begin{proof}
The energy balance of \eqref{eq:KSE} is given by
    \begin{align}\label{eq:energy:balance}
        \frac{1}2\frac{d}{dt}\no{L^2}{\bu}^2+\no{L^2}{\De\bu}^2=-\lb\bu\cdotp\nabla\bu,\bu\rb-\lam\lb\De\bu,\bu\rb.
    \end{align}
Observe that
    \begin{align}\label{eq:ibp}
        -\lb\bu\cdotp\nabla\bu,\bu\rb=\frac{1}2\lb \de,|\bu|^2\rb=\frac{1}2\lb\de_+,|\bu|^2\rb-\frac{1}2\lb\de_-,|\bu|^2\rb.
    \end{align}
On the other hand, by the Cauchy-Schwarz inequality and Young's inequality
    \begin{align}\label{est:destabilizing}
        \lam|\lb\De\bu,\bu\rb|\leq\lam\no{L^2}{\De\bu}\no{L^2}{\bu}\leq \frac{1}2\no{L^2}{\De\bu}^2+\frac{\lam^2}2\no{L^2}{\bu}^2.
    \end{align}
It follows that
    \begin{align}\label{est:energy:balance:div}
        \frac{d}{dt}\no{L^2}{\bu}^2+\no{L^2}{\De\bu}^2+\lb\de_-,|\bu|^2\rb\leq \left(\de_+^*+\lam^2\right)\no{L^2}{\bu}^2,
    \end{align}
where $\de^*_+$ is defined by \eqref{def:div:plus:max}. An application of Gr\"onwall's inequality then implies
    \begin{align}\label{est:L2:div}
        \no{L^2}{\bu(t)}^2
        \leq \exp\left(\lam^2t+\int_0^t\de_+^*(s)\ ds\right)\no{L^2}{\bu_0}^2.
    \end{align}
In particular $\sup_{t\in[0,T^*)}\no{L^2}{\bu(t)}<\infty$ provided that \eqref{cond:div} holds. By \cref{thm:exist}, we may continue the solution past $T^*$ and in particular deduce that $\bu(t)\in C^\infty$ for all $t>0$.

Conversely, suppose that \eqref{cond:div} fails. By \cref{thm:exist}, $\de_+(t,x)<\infty$, for all $t\in(0,T^*)$. Thus, $\sup_{x\in\Om}\de_+(T^*,x)=\infty$. It follows from the Sobolev embedding theorem that
    \begin{align}\notag
        C\no{H^3}{\bu(t)}
        \geq C\no{H^2}{\de(t)}
        \geq\sup_{x\in\Om}\de_+(t,x).
    \end{align}
Upon taking the limit as $t\goesto T^*-$, we deduce that $\no{H^3}{\bu(T^*)}$ must be infinite.
\end{proof}

\begin{Rmk}\label{rmk:dimension}
We point out that this regularity criterion is independent of the dimension of the spatial domain. In particular, \cref{thm:div} holds in $d=3$ as well, with the converse statement adjusted accordingly from the higher-dimensional Sobolev embedding.

\cref{thm:div} is an extension (in the $d=2$ case) of one of the regularity criteria developed in \cite[Theorem 3.13]{Larios_Rahman_Yamazaki_2021_JNLS_KSE_PS} in terms of the divergence. Namely, rather than a criterion on the whole divergence, here, we have a criterion based only on the positive part, which has a physical significance discussed above in \cref{sect:intro}.  The main thrust of \cref{thm:div} is to revisit the role of the divergence in light of the observations made in \cref{sect:Burgers}; in the next section we further refine \cref{thm:div} having in mind the understanding that the behavior of solutions to \eqref{eq:KSE} is effectively determined by its behavior on the large-scales alone.
\end{Rmk}

\subsection{A low-mode regularity criterion}\label{sect:div:low}

In this section, we establish a stronger version of \cref{thm:div} that exploits the idea that growth of the solution is entirely determined by its growth on large-scales. To state the main result, we introduce the notation $P_N\bv$ to denote the projection of $\bv$ onto its Fourier series up to wavenumbers $|k|\leq N$. We will make use of the shorthand $\bv_N=P_N\bv$ and $\bv^N=(I-P_N)\bv$, and similarly for scalar-valued functions.

Also recall that \eqref{eq:KSE} possesses the symmetry of Galilean invariance, i.e., if $\bu(t,\bx)$ is a solution to \eqref{eq:KSE} corresponding to initial data $\bu(0,\bx)=\bu_0(\bx)$, then 
	\begin{align}\label{def:galilean}
		(G_\bv\bu)(t,\bx)=\bu(t,\bx+\int_0^t\bv(s)ds)-\bv(t),
	\end{align}
is a solution of \eqref{eq:KSE} corresponding to initial data $(G_\bv\bu)(0,\bx)=\bu_0(\bx)-\bv(\add{0})$,  for any sufficiently smooth function $\bv:[0,T^*)\goesto \RR^2$, where $T^*>0$ is corresponding maximal time of existence. In particular, we may choose $\bv=\bbu$, where $\bbu$ is defined as the spatial mean of $\bu$:
    \begin{align}\label{def:mean}
        \bbu(t):=\frac{1}{|\Om|}\int_{\Om}\bu(t,\bx) d\bx.
    \end{align}
Thus, one automatically has that $\add{G_{\bar{\bu}}\bu(t)}$ is mean-free for as long as the solution $\bu$ exists. In the analysis of the main results below, it will be useful to invoke the Galilean invariance; its utility will be clear in the proofs. Before we proceed to the proof {\rtwo{of}} the main results, we collect a few elementary results regarding the relation between the Galilean transformation, Sobolev norms, mean-values, and fluctuations.

\begin{Lem}\label{lem:galilean}
Let $\bu$ be a smooth, periodic vector field over $\Om$ and $\bv:(0,\infty)\goesto\RR^2$ be locally integrable. Then
	\begin{align}\label{eq:fluctuation}
		\no{H^\kap}{G_\bv\bu(t)}^2=
		\begin{cases}\no{H^\kap}{\bu(t)}^2&\kap\neq0\\
		\no{H^\kap}{\bu(t)}^2-|\bbu(t)\add{-\bv(t)}|^2&\kap=0.
		\end{cases}
	\end{align}
\add{In particular}
	\begin{align}\label{eq:galilean:identity}
		\no{H^\kap}{G_{\bbu}\bu(t)}=\no{H^\kap}{\bu(t)-\bbu(t)},
	\end{align}
for all $\kap\in\RR$. 
\end{Lem}

\begin{proof}
For all $\ell\in\ZZ^2\smod\{\mathbf{0}\}$, we see that $\hat{\bv}_\ell(t)=0$; upon changing variables and using the translation invariance of Lebesgue measure on $\Om$, we see that
	\begin{align}\label{eq:galilean}
		\widehat{G_{\bv}\bu}_\ell(t)&=\frac{1}{(2\pi)^2}\int_{\Om}\left(\bu(t,\bx+\int_0^t\bv(s)ds)-\bv(t)\right) e^{-i\ell\cdotp\bx}d\bx
		\\
		&=e^{i\ell\cdotp\int_0^t\bv(s)ds}\left(\frac{1}{(2\pi)^2}\int_{\int_0^t\bv(s)ds+\Om}\bu(t,\bx) e^{-i\ell\cdotp\bx}d\bx\right)\notag
		\\
		&=e^{i\ell\cdotp\int_0^t\bv(s)ds}\hat{\bu}_\ell(t).\notag
	\end{align}
It follows that $|\widehat{G_\bv\bu}_\ell(t)|=|\hat{\bu}_\ell(t)|$, for all $\ell\in\ZZ^2\smod\{\mathbf{0}\}$. Hence
	\begin{align}\notag
		\no{H^\kap}{G_\bv\bu(t)}^2=
		\begin{cases}\no{H^\kap}{\bu(t)}^2&\kap\neq0\\
		\no{H^\kap}{\bu(t)}^2-|\bbu(t)\add{-\bv(t)}|^2&\kap=0.
		\end{cases}
	\end{align}

Now consider the special case $\bv=\bbu$ and let $\tbu=\bu-\bbu$. From \eqref{eq:galilean}, we also observe that
	\begin{align}\notag
		\widehat{G_{\bbu}\bu}_\ell(t)=e^{i\ell\cdotp\int_0^t\bbu(s)ds}\hat{\tbu}_\ell(t).
	\end{align}
In particular, $|\widehat{G_{\bbu}\bu}_\ell(t)|=|\hat{\tbu}_\ell(t)|$, for all $\ell\in\ZZ^2\smod\{\mathbf{0}\}$. Since $\bar{\tbu}=0$ by definition, it follows that
	\begin{align}\notag
		\no{H^\kap}{G_{\bbu}\bu(t)}=\no{H^\kap}{\tbu(t)},
	\end{align}
for all $\kap\in\RR$. This completes the proof.
\end{proof}

We will also make use of the following observation which effectively asserts the control of the mean-value of solutions to \eqref{eq:KSE} by their fluctuation component.

\begin{Lem}\label{lem:mean:value}
Let $\bu$ denote a smooth solution to \eqref{eq:KSE} corresponding to initial value $\bu_0$ and let $T^*>0$ denote its maximal time of existence. Then
	\begin{align}\label{est:mean:value}
		\sup_{0\leq t\leq T}|\bbu(t)|\leq |\bbu_0|+\frac{1}{\rtwo{8\pi^2}}\int_0^T\no{H^1}{\tbu(s)}^2ds,
	\end{align}
for any $T<T^*$.
\end{Lem}

Now let us prove \cref{lem:mean:value}.

\begin{proof}[Proof of \cref{lem:mean:value}]
Let $\tbu=\bu-\bbu$. Observe that $\bu\cdotp\nabla \bu=\tbu\cdotp\nabla\tbu+\bbu\cdotp\nabla\tbu$. Upon integrating \eqref{eq:KSE} over $\Om$, dividing by $|\Om|$, and invoking the divergence theorem and periodicity, we obtain
	\begin{align}\label{eq:mean}
		\frac{d}{dt}\bbu=\frac{1}{|\Om|}\int \tbu\cdotp\nabla\tbu d\bx.
	\end{align}
Upon integrating in time over $[0,t]$, we then obtain
	\begin{align}\notag
		\bbu=\bbu_0+\int_0^t\frac{1}{|\Om|}\int\tbu(s)\cdotp\nabla\tbu(s) d\bx\,{\rone{ds}}.
	\end{align}
By the Cauchy-Schwarz and Young's inequalities, we then deduce
	\begin{align}\notag
		|\bbu(t)|\leq |\bbu_0|+\frac{1}{2(2\pi)^2}\int_0^t\no{L^2}{\tbu(s,\bx)}^2ds+\frac{1}{2(2\pi)^2}\int_0^t\no{L^2}{\nabla\tbu(s,\bx)}^2ds.
	\end{align}
In particular
	\begin{align}\notag
		\sup_{0\leq t\leq T}|\bbu(t)|\leq |\bbu_0|+\frac{1}{8\pi^2}\int_0^T\no{H^1}{\tbu(s)}^2ds,
	\end{align}
for any $T<T^*$. By \cref{lem:galilean}, we deduce that
	\begin{align}\notag
		\sup_{0\leq t\leq T}|\bbu(t)|\leq |\bbu_0|+\frac{1}{8\pi^2}\int_0^T\no{H^1}{G_{\bbu}\bu(s)}^2ds,
	\end{align}
 proving the desired result{\add{.}}
\end{proof}

Next, we prove a theorem inspired by the work of \cite{Cheskidov_Shvydkoy_2014}, which is in the same spirit as \cref{thm:div}, but which indicates that one only needs to focus {\rtwo{on}} the divergence, weighted preferentially on low modes as measured by a \textit{negative Sobolev norm}.  Of course, similar to the Gibbs phenomenon, projecting to low modes may slightly change the regions of positivity (and hence \cref{thm:div} is not a corollary of \cref{thm:div:low}), but this is not especially disruptive to our methods, and we can still obtain the following criterion on the positive part of the low modes of the divergence.

\begin{Thm}\label{thm:div:low}
Given $\bu_0\in L^2$, let $\bu(\cdotp;\bu_0)$ denote the unique smooth solution of \eqref{eq:KSE} over its maximal interval of existence $(0,T^*)$. Given $\al\in[0,2)$ {\add{and positive numbers $C_*$, $N_*$}}, define
    \begin{align}\label{def:N}
        N_\al(t;\bu_0)=\left[C_*\left(\no{H^{-\al}}{\bu(t;\bu_0)}^2+N_*\right)\right]^{\frac{1}{2-\al}}.
    \end{align} 
For each $\al\in[0,2)$, there exists a universal constant $C_*$, independent of $\bu_0$ and $T^*$, such that {\add{for any $N_*$ such that $C_*N_*\geq1$,}} if
    \begin{align}\label{cond:div:low}
        \limsup_{t\goesto T^*-}\int_0^t(P_{N_\al(s;\bu_0)}\de)_+^*(s)\ ds<\infty,
    \end{align}
then $\sup_{t\in[0,T^*)}\no{L^2}{\bu(t)}<\infty$ and, subsequently, $T^*=\infty$. Conversely, if \eqref{cond:div:low} fails, then $\limsup_{t\goesto T^*-}\no{H^3}{\bu(t)}$ is infinite.
\end{Thm}

\begin{proof}
Throughout the proof, it will be convenient to suppress the dependence of $N_\al(t;\bu_0)$ on $t$ and $\bu_0$. For convenience, we denote $N=N_\al$. We first consider the case where $\bu$ is mean-free throughout its interval of existence. We write the cubic term in the energy balance \eqref{eq:energy:balance} as
    \begin{align}
        &-\lb (\bu\cdotp\nabla)\bu,\bu\rb=-\lb (\bu_N\cdotp\nabla)\bu,\bu\rb-\lb (\bu^N\cdotp\nabla)\bu,\bu\rb=:I+II.\notag
    \end{align}

For $I$, we argue as in \eqref{eq:ibp} {\rtwo{to obtain}}
    \begin{align}\notag
        I&= \frac{1}2\lb P_N\de,|\bu|^2\rb= \frac{1}2\lb (P_N\de)_+,|\bu|^2\rb-\frac{1}2\lb (P_N\de)_-,|\bu|^2\rb.
    \end{align}

For $II$, we make use of the fact that $\bu^N$ is mean-free. Let $\eps\in(0,1)$ and $1/p_j+1/q_j=1/2$, for $j=1,2$, where $p_j,q_j\in[2,\infty)$. Then by the Bernstein-type inequality \eqref{est:bernstein} and {\rone{Kato--Ponce}} inequality \eqref{est:kato:ponce} we obtain
    \begin{align}
        |II|&\leq \no{L^2}{D^{-\eps}\bu^N}\no{L^2}{D^\eps((\nabla\bu)\bu)}\notag
        \\
        &\leq \frac{C_{KP}}{N^{\be+\eps}}\no{H^\be}{\bu}\left(\no{L^{p_1}}{D^\eps\nabla\bu}\no{L^{q_1}}{\bu}+\no{L^{p_2}}{\nabla\bu}\no{L^{q_2}}{D^\eps\bu}\right),\notag
    \end{align}
where $\be\in[-\eps,2]$. By the Sobolev embedding \eqref{est:sobolev} and interpolation \eqref{est:interpolation}, we see that upon choosing $(p_1,q_1)=(2/\eps,\eps)$ and $(p_2,q_2)=(\eps,2\eps)$, we have (since $\bu$ is mean-free)
    \begin{align}
        \no{L^{p_1}}{D^\eps\nabla\bu}&\leq C_S\no{L^2}{\De\bu},\notag
        \\
        \no{L^{q_1}}{\bu}&\leq C_S\no{L^2}{D^\eps\bu}\leq C_S\no{L^2}{\De\bu}^{\frac{\eps+\al}{2+\al}}\no{L^2}{D^{-\al}\bu}^{\frac{2-\eps}{2+\al}},\notag
        \\
        \no{L^{p_2}}{\nabla\bu}&\leq C_S\no{L^2}{D^{1+\eps}\bu}\leq C_S\no{L^2}{\De\bu}^{\frac{1+\eps+\al}{2+\al}}\no{L^2}{D^{-\al}\bu}^{\frac{1-\eps}{2+\al}},\notag
        \\
        \no{L^{q_2}}{D^\kap\bu}&\leq C_S\no{L^2}{\nabla\bu}\leq C_S\no{L^2}{\De\bu}^{\frac{1+\al}{2+\al}}\no{L^2}{D^{-\al}\bu}^{\frac{1}{2+\al}},\notag
    \end{align}
for any $\al\in[0,2)$. For $\be\in[-(\al\wedge\eps),2]$, inequality \eqref{est:interpolation} yields
    \begin{align}\notag
        \no{H^\be}{\bu}\leq \no{L^2}{\De\bu}^{\frac{\al+\be}{2+\al}}\no{L^2}{D^{-\al}\bu}^{\frac{2-\be}{2+\al}}.
    \end{align}
Choosing $\al+\be=2-\eps$, we may now combine the above inequalities to deduce
    \begin{align}\notag
        |II|\leq 2\frac{C_{KP}C_S^2}{N^{\be+\eps}}\no{L^2}{\De\bu}^{2}\no{H^{-\al}}{\bu}.
    \end{align}
In particular, observe that the pair $(\al,\be)$ may be chosen among the following collection of pairs:
    \begin{align}\label{eq:choices}
        (\al,\be)\in\{(0,2-\eps),\ (\eps,2),\ (1,1-\eps),\ (2-\eps,0)\}_{\eps\in(0,1)}.
    \end{align}

Upon returning to \eqref{eq:energy:balance}, applying \eqref{est:destabilizing}, the estimates for $I, II$, and the definition of $N$ \eqref{def:N}, we arrive at
    \begin{align}\notag
        \frac{d}{dt}\no{L^2}{\bu}^2+\no{L^2}{\De\bu}^2+\lb\de_-,|\bu|^2\rb\leq \left(\de_+^*+\lam^2\right)\no{L^2}{\bu}^2+\frac{2C_{KP}C_S^2}{C_*}\no{L^2}{\De\bu}^{2},
    \end{align}
provided that $\be\neq-\eps$. Finally, by choosing, $C_*=2C_{KP}C_S^2\vee1$, we ensure that \eqref{est:L2:div} holds. This completes the proof for the mean-free case.

If $\bu$ is not mean-free, then we consider the Galilean shifted solution, $G_{\bbu}\bu$, satisfying \eqref{eq:KSE} with initial value $\bu_0-\bar{\bu}_0$. By \cref{lem:galilean}, the result then holds for $N_\al$ given by
    \begin{align}\label{def:N:fluc}
        N_\al(t;\bu_0)=\left[C_*\left(\no{H^{-\al}}{\tbu(t;\bu_0)}+N_*\right)\right]^{\frac{1}{2-\al}}.
    \end{align}
Note that since $C_*N_*\geq1$, we have $N_\al\geq1$.

In particular, the result asserts that if \eqref{cond:div:low} holds with $N_\al$ given by \eqref{def:N:fluc}, then $\sup_{t\in[0,T^*]}\no{L^2}{\tbu(t;\tbu_0)}<\infty$. By \cref{thm:exist}, it follows that $\sup_{t\in[0,t_0]}\no{L^2}{\bu(t;\bu_0)}<\infty$ and $\sup_{t\in[t_0,T^*]}\no{H^1}{\tbu(t;\tbu_0)}<\infty$, for some $0<t_0<T^*$. Since $\no{L^2}{\bu(t;\bu_0)}^2=\no{L^2}{\tbu(t;\tbu_0)}^2+|\bbu(t;\bu_0)|^2$, we then see from \cref{lem:mean:value} that $\sup_{t\in[0,T^*]}\no{L^2}{\bu(t;\bu_0)}<\infty$. Finally, we apply \cref{lem:mean:value} once again to replace \eqref{def:N:fluc} with \eqref{def:N}. This completes the proof.
\end{proof}

\begin{Rmk}\label{rmk:thm:div:low}
Regularity criteria in terms of the divergence were developed in \cite[Theorem 3.12]{Larios_Rahman_Yamazaki_2021_JNLS_KSE_PS} in arbitrary dimension $d$. In particular, it is shown there that in the case $d=2$ that if $\de\in L^r(0,T;L^p)$, where $p\in[1,\infty]$, $r\in[1,2]$, then $\bu$ is a globally regular solution of \eqref{eq:KSE}. \cref{thm:div:low} therefore improves upon this regularity criterion when $p=\infty$, $r=1$ by only imposing an integrability condition on the positive part of the divergence on sufficiently many low modes. The number of low-modes that need to be tracked depend on $H^{-\al}$, where $\al\in[0,2)$, which is strictly weaker than the energy-norm, $L^2$, when $\al>0$.
\end{Rmk}

Lastly, we address the integrability-in-time of the frequency cut-off $N=N(t;\bu_0)$.

\begin{Thm}\label{thm:N}
Given $\bu_0\in L^2$, let $\bu(\cdotp;\bu_0)$ denote the corresponding unique smooth solution of \eqref{eq:KSE} over its maximal interval of existence $(0,T^*)$. Given $\al\in[0,2)$, let $N_\al(\cdotp;\bu_0)$ be given as prescribed by \cref{thm:div:low}. Then $\bu(\cdotp;\bu_0)$ is globally regular if and only if $N_\al(\cdotp;\bu_0)\in L^{4}(0,T^*)$. 
\end{Thm}

\begin{proof}
Assume that $\bu$ is mean-free over its maximal interval of existence. We recall  \eqref{eq:energy:balance} and rather than apply \eqref{eq:ibp}, we simply estimate the right-hand side directly with H\"older's inequality, interpolation, \eqref{est:Ladyzhenskaya}, and Young's inequality to obtain
	\begin{align}\notag
		|\lb \bu\cdotp\nabla\bu,\bu\rb|&\leq\no{L^4}{\bu}^2\no{L^2}{\nabla\bu}\leq C\no{L^2}{\nabla\bu}^2\no{L^2}{\bu}\notag
            \\
            &\leq C\no{L^2}{\De\bu}\no{L^2}{\bu}^2\leq \frac{1}4\no{L^2}{\De\bu}^2+C\no{L^2}{\bu}^4.\notag
	\end{align}
Also    
        \begin{align}
             C\no{L^2}{\De\bu}\no{L^2}{\bu}^2&\leq C\no{L^2}{\De\bu}\no{L^2}{\bu}^{2-2\eps}\no{L^2}{\De\bu}^{\frac{2\al\eps}{2+\al}}\no{L^2}{D^{-\al}\bu}^{\frac{4\eps}{2+\al}}\notag
             \\
             &= C\no{L^2}{\De\bu}^{\frac{2+(1+2\eps)\al}{2+\al}}\no{L^2}{\bu}^{2-2\eps}\no{H^{-\al}}{\bu}^{\frac{4\eps}{2+\al}}\notag
             \\
             &\leq \frac{1}2\no{L^2}{\De\bu}^2+C\no{L^2}{\bu}^{\frac{4(1-\eps)(2+\al)}{2+\al-2\eps\al}}\no{H^{-\al}}{\bu}^{\frac{8\eps}{2+\al-2\eps\al}}.\notag
        \end{align}
Then for any $\al\in[0,2)$ and $\eps\in(0,1)$ such that $2+\al=4\eps$, we have  
        \begin{align}\notag
            C\no{L^2}{\De\bu}\no{L^2}{\bu}^2\leq\frac{1}2\no{L^2}{\De\bu}^2+C_0\no{L^2}{\bu}^{2}\no{H^{-\al}}{\bu}^{\frac{4}{2-\al}},
        \end{align}
for some absolute constant $C_0>0$. 

With \eqref{est:destabilizing}, we deduce
	\begin{align}\label{est:L2}
		 \frac{d}{dt}\no{L^2}{\bu}^2+\no{L^2}{\De\bu}^2\leq C_0\no{H^{-\al}}{\bu}^{\frac{4}{2-\al}}\no{L^2}{\bu}^2+\lam^2\no{L^2}{\bu}^2,
	\end{align}
for some universal constant $C_0$. In particular
	\begin{align}\label{est:L2:final}
		 \frac{d}{dt}\no{L^2}{\bu}^2+\no{L^2}{\De\bu}^2\leq\left(\lam\frac{C_0}{C_*} N^{4}+\lam^2\right)\no{L^2}{\bu}^2.
	\end{align}
By Gr\"onwall's inequality
	\begin{align}\notag
		\no{L^2}{\bu(t)}^2\leq \exp\left(\lam^2t+\lam\frac{C_0}{C_*}\int_0^tN^{4}(s)\ ds\right)\no{L^2}{\bu_0}^2.
	\end{align}
This establishes the result for the mean-free case. We may now argue as we did in the proof of \cref{thm:div:low} to complete the argument for the {\rtwo{non-}}mean-free case.
\end{proof}

\begin{Rmk}\label{rmk:LPS}
When $\al=0$, the condition $N_{0}(\cdotp;\bu_0)\in L^4(0,T^*)$ is consistent with the {\rone{{\rone{Ladyzhenskaya--Prodi--Serrin}}}} (LPS) type regularity criterion for \eqref{eq:KSE} that was developed in \cite{Larios_Rahman_Yamazaki_2021_JNLS_KSE_PS}. Indeed, in the case $d=2$, \cite[Theorem 3.1]{Larios_Rahman_Yamazaki_2021_JNLS_KSE_PS} implies that if $u\in L^r(0,T;W^{m,p})$, where $m\in(0,1)$, $p\in [1,2/m)$, and $r\in (4/3, 4/(1+m)]$ or else $m=0$ and $p\in(1,\infty]$, $r\in[4/3,4)$, then $u$ is a globally regular solution. Thus, $N_0(\cdotp;\bu_0)\in L^4(0,T^*)$ is simply the assertion that $\bu\in L^2(0,T^*;L^2)$. When $\al\in(0,2)$, \cref{thm:N} therefore provides a non-trivial extension of the LPS regularity criterion that includes negative-Sobolev norms, that is, by allowing for the case $m<0$ (in the notation of \cite{Larios_Rahman_Yamazaki_2021_JNLS_KSE_PS}). Similar results can of course be obtained for the case $d=3$; we restrict our discussion to the $d=2$ for narrative clarity. 

Intuitively speaking, the growth of negative Sobolev norms indicate a growth of energy in low frequencies, and hence, large-scales, since negative Sobolev norms penalize energy in high frequencies. Thus, \cref{thm:div:low} provides yet another alternative perspective to the issue of regularity of \eqref{eq:KSE} that refines the statement that the behavior of the solution on large-scales controls the behavior of the solution on all scales. 
\end{Rmk}

\begin{Rmk}\label{rmk:sharpness}
We point out that the range of $\al$ is almost sharp in the sense that we can almost reach the ``critical threshold'' $\al=2$. Indeed one observes the following critical-type phenomenon when estimating the trilinear interaction in the energy balance:
    \begin{align}\label{est:critical}
        |\lb G_{\bbu}\bu\cdotp\nabla G_{\bbu}\bu,G_{\bbu}\bu\rb|\leq C_L\no{L^2}{\De G_{\bbu}\bu}^2\no{H^{-2}}{G_{\bbu}\bu},
    \end{align}
which can be proved by applying H\"older's inequality, the Ladyzhenskaya inequality, and \eqref{est:interpolation}, where $C_L$ denotes the constant from the Ladyzhenskaya inequality (see also \eqref{est:Ladyzhenskaya}):
    \begin{align}\notag
        \no{L^4}{\ph}^2\leq C\no{L^2}{\nabla\ph}\no{L^2}{\ph}.
    \end{align}
In fact, in the context of $\Om=\RR^2$, one may identify $H^{-2}$ as a scaling-critical norm for the system \eqref{eq:KSE} in the particular case $\lam=0$, which possesses the scaling-symmetry, $\bu_\tht(t,\bx)=\tht^3\bu(\tht^4t,\tht\bx)$, for $\tht>0$. In particular, $\bu_\tht$ is a solution of \eqref{eq:KSE} whenever $\bu$ is a solution and $\no{H^{-2}(\RR^2)}{\bu_\tht}=\no{H^{-2}(\RR^2)}{\bu}$.

In light of these remarks, one may essentially consider our analysis as ``subcritical.'' Reaching the $\al=2$ criticality threshold would yield an optimal improvement of \eqref{thm:div:low}, \eqref{thm:N}. However, the fact that \eqref{def:N} appears to blow up in the limit as $\al\goesto 2$ suggests that there may be some interesting challenges to overcome in reaching this endpoint. Indeed, this endpoint case may be viewed as an analog of the celebrated $L^3$-based endpoint of the LPS-regularity criterion for the 3D {\rone{Navier--Stokes}} equations established by Escauriaza, Seregin, and \v Sver\'ak \cite{Escauriaza_Seregin_Sverak_2003}.

It would therefore be interesting to study whether the $H^{-2}$ norm or other scale-consistent quantity, especially those which are $L^p$--based, can serve as the frequency cut-off characterized by \eqref{def:N} or to study the potential for ill-posedness at this regularity threshold.
\end{Rmk}

\section{Global regularity of the 2D Castrated KSE}\label{sect:castrate}

Motivated by the results of \cref{sect:div:low}, we propose a modification of \eqref{eq:KSE} that inhibits the transfer of energy to large-scales. Indeed, we observe from the energy balance \eqref{eq:energy:balance} that the trilinear interaction term can be expanded as
    \begin{align}
        &-\lb (\bu\cdotp\nabla)\bu,\bu\rb\notag\\
        &=-\lb P_N((\bu\cdotp\nabla)\bu),\bu\rb-\lb (I-P_N)((\bu\cdotp\nabla)\bu), \bu\rb\notag\\
        &=-\lb(\bu_N\cdotp\nabla)\bu_N,\bu_N\rb-\lb(\bu_N\cdotp\nabla)\bu^N,\bu_N\rb-\lb (\bu^N\cdotp\nabla)\bu,\bu_N\rb+II\notag\\
        &=I_a+I_b+I_c+II\notag,
    \end{align}
where
	\begin{align}\notag
		II:=-\lb (I-P_N)((\bu\cdotp\nabla)\bu), \bu\rb.
	\end{align}
We therefore propose the following modification of \eqref{eq:KSE}:
    \begin{align}\label{eq:KSE:castrate}
        \bdy_t\bu&+\De^2\bu+\lam\De\bu\notag
        \\
        &=-((\bu_{N(\bu)}\cdotp\nabla)\bu^{N(\bu)})_{N(\bu)}-((\bu^{N(\bu)}\cdotp\nabla)\bu)_{N(\bu)}-((\bu\cdotp\nabla)\bu)^{N(\bu)},
    \end{align}
where $N(\bu)$ is defined by the functional
    \begin{align}\label{def:Nu}
         N(\bu)=C_*\left(\no{L^2}{\bu}+N_*\right),
    \end{align}
for some suitable number $N_*$ and positive constant $C_*$ such that $C_*N_*\geq {\rtwo{1}}$. Note that this choice of $N_*, C_*$ ensures that $N(\bu)\geq1$. Observe that the first three terms on the right-hand side of the balance are precisely $I_b, I_c$, and $II$. In particular, the mechanism that would have produced large-scale energy from \textit{exclusively} large-scale interactions has been culled in \eqref{eq:KSE:castrate}. For this reason, we refer to \eqref{eq:KSE:castrate} as the \textit{castrated KSE}. We show below that this system is globally regular. 

\begin{Thm}\label{thm:castrate}
Let $\bu_0\in L^2$ and $\lam>0$. There exist $C_*, N_*$ such that for all constants $C\geq C_*$ and $N\geq N_*$ defining \eqref{def:Nu}, a unique solution $\bu\in C_w([0,T];L^2))\cap L^2(0,T;H^2)$ of \eqref{eq:KSE:castrate} corresponding to initial data $\bu(0)=\bu_0$ exists, for all $T>0$. Moreover, for any $t_0>0$ and $k\geq1$, one has $\bu\in C([t_0,T];H^k)$ such that
    \begin{align}\notag
        \sup_{t\in[t_0,T]}\no{H^k}{\bu(t)}\leq C(t_0, T,\no{L^2}{\bu_0})<\infty.
    \end{align}
In particular, $\bu(t)\in C^\infty$, for all $t\in(0,T]$. Lastly, if $\nabla^\perp\cdotp\bu_0=0$, then $\nabla^\perp\cdotp\bu(t)=0$, for all $t>0$.
\end{Thm}

\begin{proof}
Due to the fact that \eqref{eq:KSE:castrate} is a reduced version of \eqref{eq:KSE}, the same analysis as the one outlined in \cref{sect:appendix} for \eqref{eq:KSE} can be carried out \textit{mutatis mutandis} for the Galerkin approximation of \eqref{eq:KSE:castrate}. The only technicality that needs to be addressed is the issue of ensuring that the frequency cut-off \eqref{def:Nu} remains well-defined throughout the evolution of the system. For this, we truncate \eqref{def:Nu} by $N^{(n)}(\bu):=N(\bu)\wedge n$, so that $N^{(n)}\leq n$. The same analysis produces a sequence of solutions $\{\bu^{(n)}\}_{n\geq1}$ corresponding to the system defined by the truncated cut-offs $N^{(n)}$, for which all of our apriori estimates hold uniformly in $n$, allowing for passage to the limit (via a diagonalization argument) as $n\goesto\infty$. Thus, under the same hypothesis as \eqref{thm:exist}, one has local existence of smooth solutions to \eqref{eq:KSE:castrate} emanating from initial data $\bu_0\in L^2$. We omit these details and simply refer the reader to \cref{sect:appendix}. In particular, it will again suffice to control $\bu$ in $L^2$. Note that throughout the proof we will suppress the dependence of $N(\bu)$ on $\bu$ for convenience. 

The energy balance of \eqref{eq:KSE:castrate} is given by
	\begin{align}
            \frac{1}2\frac{d}{dt}\no{L^2}{\bu}^2&+\no{L^2}{\De\bu}^2\notag
            \\
            &=-\lb (\bu_{N}\cdotp\nabla)\bu^{N},\bu_N\rb-\lb (\bu^{N}\cdotp\nabla)\bu,\bu_N\rb-\lb (\bu\cdotp\nabla)\bu,\bu^N\rb-\lam\lb\De\bu,\bu\rb\notag
            \\
            &=I_b+I_c+II-\lam\lb\De\bu,\bu\rb.\notag
	\end{align}
Observe that it suffices to treat $I_b$, $I_c$, and $II$ since the fourth term was already treated in \eqref{est:destabilizing}. Denoting again $\tbu:=\bu-\bbu$, it will also be helpful to note that \eqref{def:Nu} implies
	\begin{align}\notag
		N=C_*\left((\no{L^2}{\tbu}^2+|\bbu|^2)^{1/2}+N_*\right).
	\end{align}

For $I_b$, first observe that
	\begin{align}
		I_b=-\lb (\tbu_N\cdotp\nabla)\bu^N,\bu_N\rb-\lb (\bbu\cdotp\nabla)\bu^N,\bu_N\rb=I_b^1+I_b^2\notag.
	\end{align}
{\rone{We}} apply H\"older's inequality, interpolation \eqref{est:interpolation}, the Poincar\'e inequality \eqref{est:bernstein}, \eqref{def:Nu}, and Young's inequality to estimate
    \begin{align}
        |I_b^1|&\leq \no{L^4}{\tbu_N}^2\no{L^2}{\nabla\bu^N}\notag
        \\
        &\leq \frac{C}{N}\no{H^1}{\tbu_N}\no{L^2}{\tbu_N}\no{L^2}{\De\bu^N}\notag
        \\
        &\leq \frac{C}{N}\no{L^2}{\De\tbu}^{1/2}\no{L^2}{\tbu}^{1/2}\no{L^2}{\tbu_N}\no{L^2}{\De\bu^N}\notag
        \\
        &\leq \frac{1}{12}\no{L^2}{\De\bu}^2+\frac{C}{N^4}\no{L^2}{\tbu_N}^4\no{L^2}{\bu}^2\notag
        \\
        &\leq \frac{1}{12}\no{L^2}{\De\bu}^2+\frac{C}{C_*^4}\no{L^2}{\bu}^2,\notag
    \end{align}
for some universal constant $C>0$. Also
    \begin{align}
        |I_b^2|&\leq |\bbu|\no{L^2}{\nabla\bu^N}\no{L^2}{\bu_N}\notag
        \\
        &\leq \frac{|\bbu|}{N}\no{L^2}{\De\bu^N}\no{L^2}{\bu_N}\notag
        \\
        &\leq \frac{1}{12}\no{L^2}{\De\bu}^2+\frac{C}{4\pi^2N^2}\no{L^2}{\bu_N}^2\no{L^2}{\bu}^2\notag
        \\
         &\leq\frac{1}{12}\no{L^2}{\De\bu}^2+\frac{C}{C_*^2}\no{L^2}{\bu}^2.\notag
    \end{align}
      
To treat $I_c$ we first observe that
	\begin{align}
		I_c=-\lb (\bu^N\cdotp\nabla)\tbu, \tbu_N\rb-\lb (\bu^N\cdotp\nabla)\tbu,\bbu\rb=I_c^1+I_c^2\notag.
	\end{align}
We then estimate $I_c$ similarly to $I_b$. In particular, we have
    \begin{align}
        |I_c^1|&\leq \no{L^2}{\bu^N}\no{L^4}{\nabla\tbu}\no{L^4}{\tbu_N}\notag
        \\
        &\leq \frac{C}{N^2}\no{L^2}{\De\bu^N}\no{L^2}{\De\tbu}^{3/4}\no{L^2}{\tbu}^{1/4}\no{L^2}{\nabla\tbu_N}^{1/2}\no{L^2}{\tbu_N}^{1/2}\notag
        \\
        &\leq \frac{C}{N^{3/2}}\no{L^2}{\De\bu}^{7/4}\no{L^2}{\tbu}^{1/4}\no{L^2}{\tbu_N}\notag
        \\
        &\leq \frac{1}{12}\no{L^2}{\De\bu}^2+\frac{C}{N^{12}}\no{L^2}{\tbu_N}^8\no{L^2}{\tbu}^2\notag
        \\
        &\leq \frac{1}{12}\no{L^2}{\De\bu}^2+\frac{C}{C_*^{12}N_*^4}\no{L^2}{\bu}^2,\notag
    \end{align}
for some universal constant $C>0$, where we invoked \eqref{def:N} in obtaining the final inequality. We also have
    \begin{align}
        |I_c^2|&\leq \no{L^2}{\bu^N}\no{L^2}{\nabla\tbu}|\bbu|\notag
        \\
        &\leq \frac{C}{4\pi^2N^2}\no{L^2}{\De\bu^N}\no{L^2}{\De\tbu}^{1/2}\no{L^2}{\tbu}^{1/2}\no{L^2}{\bu}\notag
        \\
        &\leq \frac{C}{N^2}\no{L^2}{\De\bu}^{3/2}\no{L^2}{\bu}^{3/2}\notag
        \\
  	&\leq \frac{1}{12}\no{L^2}{\De\bu}^2+\frac{C}{N^8}\no{L^2}{\bu}^6,\notag
        \\
	&\leq \frac{1}{12}\no{L^2}{\De\bu}^2+\frac{C}{C_*^{8}N_*^2}\no{L^2}{\bu}^2.\notag
    \end{align}

Lastly, we estimate $II$. As before, let us first observe that
	\begin{align}
		II=-\lb(\tbu\cdotp\nabla)\tbu,\bu^N\rb-\lb (\bbu\cdotp\nabla)\tbu,\bu^N\rb=II^1+II^2.\notag
	\end{align}
Observe that upon applying H\"older's inequality, interpolation, the inverse Poincar\'e inequality, and the Cauchy-Schwarz inequality, we obtain
    \begin{align}
        |II^1|&\leq \no{L^4}{\tbu}\no{L^4}{\nabla\tbu}\no{L^2}{\bu^N}\notag
        \\
        &\leq C\no{L^2}{\nabla\tbu}\no{L^2}{\tbu}^{1/2}\no{L^2}{\De\tbu}^{1/2}\no{L^2}{\bu^N}\notag
        \\
        &\leq \frac{C}{N^2}\no{L^2}{\De\bu}\no{L^2}{\De\tbu}\no{L^2}{\tbu}\notag
        \\
        &\leq \frac{C}{C_*^2N_*}\no{L^2}{\De\bu}^2,\notag
    \end{align}
for some universal constant $C>0$. Similarly
    \begin{align}
        |II^2|&\leq |\bbu|\no{L^2}{\nabla\tbu}\no{L^2}{\bu^N}\notag
        \\
        &\leq \frac{C}{2\pi N^2}\no{L^2}{\bu}\no{L^2}{\De\tbu}^{1/2}\no{L^2}{\tbu}^{1/2}\no{L^2}{\De\bu^N}\notag
        \\
        &\leq \frac{C}{N^2}\no{L^2}{\De\bu}^{3/2}\no{L^2}{\bu}\no{L^2}{\tbu}^{1/2}\notag
        \\
        &\leq \frac{1}{12}\no{L^2}{\De\bu}^2+\frac{C}{N^8}\no{L^2}{\bu}^4\no{L^2}{\tbu}^2\notag
        \\
        &\leq \frac{1}{12}\no{L^2}{\De\bu}^2+\frac{C}{C_*^8N_*^2}\no{L^2}{\bu}^2\notag,
    \end{align}
for some universal constant $C>0$.
    
Let $C_0$ denote the maximum over all constants $C$ appearing in the estimates above. We may then choose any positive constants $C_*, N_*$ satisfying
    \[
            \frac{C_0}{C_*^2N_*}\leq\frac{1}{12}\quad\text{and}\quad \frac{C_0}{C_*^2}\left(1+\frac{1}{C_*^2}+\frac{1}{C_*^6N_*^2}+\frac{1}{C_*^{10}N_*^4}\right)\leq \frac{1}2(\lam\vee1)^2.
    \]
Then upon combining the above estimates and adding $\no{L^2}{\bu}^2$ to both sides of the resulting inequality, we arrive
	\begin{align}\notag
		\frac{d}{dt}\no{L^2}{\bu}^2+\no{H^2}{\bu}^2+\leq 4(\lam\vee1)^2\no{L^2}{\bu}^2.
	\end{align}
An application of Gr\"onwall's inequality then yields
	\begin{align}\notag
		\no{L^2}{\bu(t)}^2+\int_0^te^{4(\lam\vee1)^2(t-s)}\no{H^2}{\bu(s)}^2\ ds\leq e^{4(\lam\vee1)^2t}\no{L^2}{\bu_0}^2,
	\end{align}
and we are done.
\end{proof}

\begin{Rmk}\label{rmk:restrict}
On the other hand, we may also control the transfer of energy produced via nonlinear interaction by removing \textit{both} the production of small-scale energy through nonlinear interaction in addition to the large-scale energy produced from large-scale interactions. In particular, one may consider what may be referred to as a ``cascade-restricted'' KSE-type system:
     \begin{align}\label{eq:KSE:restrict}
       \bdy_t\bu+((\bu_{N(\bu)}\cdotp\nabla)\bu^{N(\bu)})_{N(\bu)}+((\bu^{N(\bu)}\cdotp\nabla)\bu)_{N(\bu)}=-\De^2\bu-{\rtwo{\lam}}\De\bu,
       \end{align}
When $\lam=0$, one may redefine $N(\bu)=C_*\left(\no{L^2}{P_N\bu}+N_*\right)$, with similar defining constants $C_*, N_*$ as in \cref{thm:div:low}. In this case, \eqref{eq:KSE:restrict} becomes a cascade-controlled hyperviscous Burgers-type equation. The above analysis then implies the existence of a finite-dimensional global attractor. It would be interesting to study numerically whether the long-time average of the energy for this system is an intensive quantity (see the discussion in  \cref{sect:intro}).
\end{Rmk}  

\begin{Rmk}\label{rmk:control}
We observe that the system \eqref{eq:KSE:castrate} can be viewed as a controlled-KSE system. In particular, \eqref{eq:KSE:castrate} can be written as
    \begin{align}\notag
        \bdy_t\bu+\bu\cdotp\nabla\bu=-\De^2\bu-\lam\De\bu+\bbf(\bu;N),
    \end{align}
where $\bbf(\bu;N)$ is given by
    \begin{align}
        \bbf(\bu;N)=P_{N(\bu)}(\bu_{N(\bu)}\cdotp\nabla \bu_{N(\bu)}).\notag
    \end{align}
From this point of view, the control $\bbf(\bu;N)$ systematically removes large-scale/large-scale interactions that contribute directly to large-scale motions as the KSE solution evolves. An alternative approach to understanding the issue of global regularity could therefore be to study the extent to which the low-mode/low-mode interactions can {\rtwo{be}} systematically added to \eqref{eq:KSE:castrate} and retain the property of global well-posedness.
\end{Rmk}

\begin{Rmk}\label{rmk:Numerical}
At first glace, system \eqref{eq:KSE:castrate} may appear fairly complicated from the perspective of numerical simulations, since the ``cut-off'' wave number $N(\bu)$ depends on the solution, but by treating the system semi-implicitly (to avoid the stiffness of the linear terms), so that $N(\bu)$ depends on the previous time-step, a reasonable numerical scheme can be devised.  For example, for simple semi-implicit Euler time-stepping, one could consider (with $\bu^k\approx\bu(t_k)$, time-step $h>0$, and some appropriate spatial discretization):
    \begin{align}\label{eq:KSE:castrate_Euler}
        &\quad
        (1+h\De^2+h\lam\De)\bu^{k+1}
        \\ \notag
        &=\bu^k-h((\bu^k_{N(\bu^k)}\cdotp\nabla)(\bu^k)^{N(\bu^k)})_{N(\bu^k)}
        -h(((\bu^k)^{N(\bu^k)}\cdotp\nabla)\bu^k)_{N(\bu^k)}
        -h((\bu^k\cdotp\nabla)\bu^k)^{N(\bu^k)},
    \end{align}
which is straight-forward to implement.
We do not comment here on the stability or consistency of scheme \eqref{eq:KSE:castrate_Euler}, but we plan to study simulations of \eqref{eq:KSE:castrate} in a future work.
\end{Rmk}

\section{Global regularity of the 2D {\rone{Burgers--Sivashinsky}} equation}\label{sect:BSE}

In the final section of the paper, we apply \cref{thm:div} to show that solutions to the so-called 2D ``{\rone{Burgers--Sivashinsky}}'' equation (see, e.g., \cite{Goodman_1994,Rakib_Sivashinsky_1987}) are globally regular. \add{Although this result is known (see, e.g., \cite{Molinet_2000, Molinet_2000_2D_BS}), the proof we provide is new and a straightforward consequence of \cref{thm:div}, {\add{which the reader can verify is also valid for the Burgers--Sivashinsky equation.}}}

Recall that the curl-free {\rone{Burgers--Sivashinsky}} equation is given by
    \begin{align}\label{eq:BSE}
        \bdy_t\bu+\bu\cdotp\nabla\bu=\De\bu+\lam\bu,\quad \nabla^\perp\cdotp\bu=0.
    \end{align}
Note that \eqref{eq:BSE} still possesses a Galilean-type invariance. Indeed, upon letting $\bw=e^{\add{-}\lam t}\bu$, we see that
	\begin{align}\label{eq:BSE:exp}
		\bdy_t\bw+e^{\lam t}\bw\cdotp\nabla\bw=\De\bw.
	\end{align}
Thus, for $G^\lam_\bv\bw=\bw(t,\bx+\int_0^te^{\lam s}\bv(s)ds)-\bv(t)$, if $\bw$ satisfies \eqref{eq:BSE:exp} with initial data $\bw_0$, then $G^\lam_{\add{\bar{\bw}}}\bw$ satisfies \eqref{eq:BSE:exp} with \add{mean-}zero initial data, \add{where $\bar{\bw}(t)=|\Om|^{-1}\int_{\Om}\bw(t,\bx)d\bx$}. \add{We remark that analysis similar to that performed in the previous sections may then still be applied to obtain similar results for the Burgers--Sivashinsky equation. However, in what follows, only the analog of \cref{thm:div} is needed.}

\begin{Thm}
Given a smooth vector field $\bu_0$ such that $\nabla^\perp\cdotp\bu_0=0$, there exists a unique solution $\bu\in C([0,T];L^2))\cap L^2(0,T;H^1)$ of \eqref{eq:Burgers:div} with initial data $\bu(0)=\bu_0$, for all $T>0$, such that $\bu(t)\in C^\infty$, for all $t\in[0,T]$.
\end{Thm}

\begin{proof}
Let $\bu$ be the unique local-in-time solution of \eqref{eq:BSE} corresponding to initial data $\bu_0$ so that $\de=\nabla\cdotp\bu$ is also smooth. Similar to \eqref{eq:Burgers:div}, the scalar divergence satisfies 
    \begin{align}\label{eq:BSE:div}
        \bdy_t\de+\bu\cdotp\nabla\de=-|\nabla\bu|^2+\De\de+\lam\de
    \end{align}
Suppose that $\de^*(t)=\max_{x\in\Om}\de(t,x)=\de(t,x^*(t))$. {\rtwo{For $t>0$, observe that $\nabla\de|_{(t,x^*(t))}=0$ and that $\De\de|_{(t,x^*(t))}\leq0$. Then upon evaluating \eqref{eq:BSE:div} at $x^*(t)$ we obtain}}
    \begin{align}\label{est:div:star}
        \frac{d}{dt}\de^*\leq \lam\de^*.
    \end{align}
Hence $\de^*(t)\leq e^{\lam t}\de^*(0)$. {\rtwo{Since $\de^*(0)\leq(\de^*(0))_+$ and $\de^*_+(t)=0$, if $\de^*(t)\leq0$, else $\de^*(t)=\de^*_+(t)$, it follows that}} 
$\de^*_+(t)\leq e^{\lam t}(\de^*(0))_+$, 
which verifies \eqref{cond:div}.
\end{proof}

\begin{Rmk}
\rtwo{For a complete justification of interchanging the maximum operator with the time-derivative in obtaining \eqref{est:div:star}, we refer the reader to \cite[Appendix B]{ConstantinTarfuleaVicol2015}.}
\end{Rmk}

\begin{Rmk}
\add{It is shown in \cite{Molinet_2000} that an apriori bound for \eqref{eq:BSE} of a quantity stronger than $\|\de_+(t)\|_{L^\infty}$ can be obtained, namely, for the quantity $\al_p(t):=\|(\bdy_1u^1(t))_+\|_{L^p}^p+\|(\bdy_2u^2(t))_+\|_{L^p}^p$, for all $p\geq3$. This quantity is, in turn, used to obtain an control $\|\de_+(t)\|_{L^2}$, and therefore deduce global regularity. An apriori estimate for $\|\de(t)\|_{L^\infty}$, then immediately follows. However, \cite{Molinet_2000, Molinet_2000_2D_BS} achieve much more with their analysis and ultimately, sharp estimates on the absorbing ball in $L^2$ are obtained.}
\end{Rmk}

\section{Conclusions}\label{sect:conclude}

To summarize, we establish the primacy of the divergence, particularly the low-mode behavior of the positive part of the divergence, in determining global regularity of solutions to the Kuramoto--Sivashinsky equation, \eqref{eq:KSE}, in dimension $d=2$ (\cref{thm:div}, \cref{thm:div:low}) This observation is motivated from the monotonicity of the divergence in the 2D curl-free Burgers equation, \eqref{eq:Burgers} in Lagrangian coordinates. From this point of view, our results suggest that an analysis of the interplay between the stabilizing and de-stabilizing mechanisms present in \eqref{eq:KSE} should also account for the sign of the divergence.

The cut-off frequency present in our divergence-based regularity criterion provides a unified approach to regularity in the spirit \cite{Cheskidov_Shvydkoy_2014}. This approach allows for more expansive regularity criterion that complements existing results in the literature (\cref{thm:N}), specifically the rather comprehensive study \cite{Larios_Rahman_Yamazaki_2021_JNLS_KSE_PS}. In particular, we identify the role of negative Sobolev norms as a type of \textit{critical quantity} that warrants further investigation (see \cref{rmk:LPS} and \cref{rmk:sharpness}).

Our study of the low-mode behavior of the divergence further inspires a modification to \eqref{eq:KSE}, namely, the castrated KSE \eqref{eq:KSE:castrate}. The castrated KSE removes large-scale interactions that contribute to amplification of large-scale energy, but retains all other interactions. Provided that a sufficiently large number of these interactions are removed, we are able to establish global regularity of solutions. It would be interesting to study the long-time behavior of \eqref{eq:KSE:castrate}, investigate whether or not it is exhibits chaotic dynamics, as well as study how it compares to the original 2D KSE model \eqref{eq:KSE}. However, since we do not establish a uniform-in-time bound of the energy for this system, we propose a further modification of \eqref{eq:KSE:castrate} that possesss and $L^2$ absorbing ball, which we refer to as the cascade-restricted KSE \eqref{eq:KSE:restrict}. This further modified KSE may be closer to the 1D KSE and potentially frutiful ground for studying whether or not the long-time averaged energy is an \textit{intensive quantity} in the sense described in the introduction.  We believe these issues are interesting and deserve further investigation, both analytically and computationally, for example using a scheme in the spirit of the scheme we proposed in Remark \ref{rmk:Numerical}. 

Lastly, as an application our regularity criterion, we supply an efficient proof of the global regularity of the 2D Burgers-Shivashinsky model \eqref{eq:BSE}. It is our hope that our regularity criteria may also shed light  on the issue of global regularity for the 2D Michelson-Shivashinsky model, whose vector form is given by
    \begin{align}\label{eq:MSE}
        \bdy_t\bu+\bu\cdotp\nabla\bu=\De\bu+\lam(-\De)^{1/2}\bu.
    \end{align}
In this direction, we specifically refer the reader to a recent result of H. Ibdah \cite{Ibdah_2021_MichelsonSivashinsky}, wherein it is shown that a modified \eqref{eq:MSE} is shown to be globally regular in any dimension.

\subsection*{Data Availability} We do not analyse or generate any datasets, because our work proceeds within a theoretical and mathematical approach. One can obtain the relevant materials from the references below.

\subsection*{Ethics Declarations} 
On behalf of all authors, the corresponding author states that there is no conflict of interest.

\subsection*{Acknowledgments} We would like to thank the anonymous reviewers for their insightful comments and suggestions, which have improved the manuscript. We would also like to thank David Ambrose, Theo Drivas, and Aseel Farhat for stimulating discussions related to this work. A.L. was supported in part by NSF grants DMS-2206762 and CMMI-1953346, and USGS Grant No. G23AC00156-01. V.R.M. was in part supported by the National Science Foundation through DMS 2213363 and DMS 2206491, as well as the Dolciani Halloran Foundation.

\appendix

\section{Local Existence of solutions}\label{sect:appendix}

We provide the relevant energy estimates that ultimately imply local existence and uniqueness of strong solutions. We proceed in a formal fashion and remark that a rigorous argument can be made by carrying out the energy estimates for the corresponding Galerkin approximation; one may observe that the estimates performed below will yield estimates uniform in the dimension of the Galerkin system.

It will be convenient to have the following interpolation inequalities on hand: 
	\begin{align}
		\no{L^4}{\bv}^2&\leq C\no{H^1}{\bv}\no{L^2}{\bv},\label{est:Ladyzhenskaya}\\
		\no{L^\infty}{\bdy^\al\bv}
		&\leq C\no{L^2}{\De\bdy^\al\bv}^{\frac{2|\al|+2}{2(|\al|+2)}}\no{L^2}{\bv}^{\frac{1}{|\al|+2}}\label{est:Agmon},
	\end{align}
where $\al$ is any multi-index such that $|\al|\geq1$.

\begin{proof}[Proof sketch of \cref{thm:exist}]
First, we establish estimates in $L^2$. We recall the energy balance \eqref{eq:energy:balance} from the proof of \cref{thm:div}. We alternatively estimate the trilinear term using H\"older's inequality, \eqref{est:interpolation}, and \eqref{est:Ladyzhenskaya}:
    \begin{align}
        -\lb \bu\cdotp\nabla\bu,\bu\rb&\leq C\no{H^1}{\bu}^2\no{L^2}{\bu}\notag
        \\
        &\leq C(\no{L^2}{\nabla\bu}^2+\no{L^2}{\bu}^2)\no{L^2}{\bu}\notag
        {\rone{= C\no{L^2}{\nabla\bu}^2\no{L^2}{\bu}+C\no{L^2}{\bu}^3}}\notag
        \\
        &\leq C\no{L^2}{\De\bu}\no{L^2}{\bu}^2+C\no{L^2}{\bu}^3\notag
        \\
        &\leq \frac{1}4\no{L^2}{\De\bu}^2+C(1+\no{L^2}{\bu}^2)\no{L^2}{\bu}^2.\notag
    \end{align}
Combining this with \eqref{est:destabilizing}, then adding $\frac{1}4\no{L^2}{\bu}^2$ to both sides yields
    \begin{align}
        \frac{d}{dt}\no{L^2}{\bu}^2+\frac{1}4\no{H^2}{\bu}^2\leq C(1+\lam+\no{L^2}{\bu}^2)\no{L^2}{\bu}^2.\notag
    \end{align}
It then follows from Gr\"onwall's inequality that there exists $T^*>0$ such that for all $T<T^*$, $\bu$ satisfies
	\begin{align}\label{est:L2:local}
		\sup_{0\leq t\leq T}\left(\no{L^2}{\bu(t)}^2+\int_0^{t}\no{H^2}{\bu(s)}^2\ ds\right)\leq C_*(T,\no{L^2}{\bu_0}),
	\end{align}
for all $t\in[0,T]$, where $\lim_{T\goesto T^*-}C_*(T,\no{L^2}{\bu_0})=\infty$. In particular, $\bu\in L^2(0,T;H^2)$, for all $T<T^*$, so that $\bu(t)\in H^2$ for a.e. $t\in(0,T^*)$. We moreover estimate
    \begin{align}
        \int_0^T\no{H^{-2}}{\frac{d\bu}{dt}(t)}^2dt
        &\leq C\int_0^T\no{H^2}{\bu(t)}^2dt+C\int_0^T\no{L^2}{\bu(t)\cdotp\nabla\bu(t)}^2dt+C\lam^2\int_0^T\no{L^2}{\bu(t)}^2dt\notag
        \\
        &\leq C\int_0^T\no{H^2}{\bu(t)}^2dt+C\int_0^T\no{H^2}{\bu(t)}^2\no{L^2}{\bu(t)}^2+C\lam^2\int_0^T\no{L^2}{\bu(t)}^2dt\notag
        \\
        &\leq C\left[1+\left(\sup_{t\in[0,T]}\no{L^2}{\bu(t)}\right)^2\right]\int_0^T\no{H^2}{\bu(t)}^2dt+C\lam^2\int_0^T\no{L^2}{\bu(t)}^2dt.\notag
    \end{align}
Hence $\frac{d\bu}{dt}\in L^2(0,T;H^{-2})$.

We now deduce $H^k$ estimates for \eqref{eq:KSE}, for any $k\geq2$. 
Let $\al\in(\NN\cup\{0\})^k$, where $k\geq1$, and $|\al|=k$. The corresponding $H^k$ balance is obtained by taking the $L^2$ inner produce of \eqref{eq:KSE} with $\bdy^{2\al}\bu$, and then sum over all $|\al|=k$. Indeed, we obtain
	\begin{align}\notag
		\frac{1}2\frac{d}{dt}\no{H^k}{\bu}^2+\no{H^k}{\De\bu}^2=-(-1)^k\sum_{|\al|=k}\lb \bdy^\al(\bu\cdotp\nabla)\bu,\bdy^\al\bu\rb-(-1)^k\lam\sum_{|\al|=k}\lb\bdy^{\al}\De\bu,\bdy^\al\bu\rb
	\end{align}
Observe that
	\begin{align}
		&\lb\bdy^\al(\bu\cdotp\nabla\bu),\bdy^{\al}\bu\rb\notag\\
		&=\sum_{\de+\be=\al}c_{\de,\be}\lb\bdy^{\de}\bu^j\bdy_j\bdy^{\be}\bu^\ell\bdy^\al\bu^\ell\rb\notag\\
		&=c_{\al,0}\lb \bdy^{\al}\bu^j\bdy_j\bu^\ell,\bdy^\al\bu^\ell\rb+c_{0,\al}\lb\bu^j\bdy_j\bdy^\al\bu^\ell\bdy^\al\bu^\ell\rb+\sum_{0<|\de|,|\be|<k}c_{\de,\be}\lb \bdy^\de\bu^j, \bdy_j\bdy^\be\bu^\ell\bdy^\al\bu^\ell\rb.\notag
	\end{align}
 Since the derivatives are mean-free due to the periodic boundary conditions, by interpolation, \eqref{est:Agmon}, and Young's inequality, one has
	\begin{align}
		|\lb \bdy^{\al}\bu^j\bdy_j\bu^\ell,\bdy^\al\bu^\ell\rb|&=\no{L^2}{\bdy^\al\bu}^2\no{L^\infty}{\nabla\bu}\notag\\
		&\leq C\no{L^2}{\De\bdy^\al\bu}^{\frac{2k}{k+2}}\no{L^2}{\bu}^{\frac{4}{k+2}}\no{L^2}{\De\nabla\bu}^{1/2}\no{L^2}{\nabla\bu}^{1/2}\notag\\
		&\leq C\no{L^2}{\De\bdy^\al\bu}^{\frac{2(k+1)}{k+2}}\no{L^2}{\bu}^{\frac{k+4}{k+2}}\notag\\
		&\leq \frac{1}8\no{L^2}{\De\bdy^\al\bu}^2+C\no{L^2}{\bu}^{\frac{k+4}{k+1}}.\notag
	\end{align}
Also, upon integrating by parts and estimating as before, we see that
	\begin{align}
		|\lb\bu^j\bdy_j\bdy^\al\bu^\ell\bdy^\al\bu^\ell\rb|&=|\lb \de,|\bdy^\al\bu|^2\rb|\notag\\
		&\leq \no{L^\infty}{\nabla\bu}\no{L^2}{\bdy^\al\bu}^2\notag\\
		&\leq \frac{1}8\no{L^2}{\De\bdy^\al\bu}^2+C\no{L^2}{\bu}^{\frac{k+4}{k+1}}.\notag
	\end{align}
Lastly
	\begin{align}
		|\lb \bdy^\de\bu^j, \bdy_j\bdy^\be\bu^\ell\bdy^\al\bu^\ell\rb|&\leq \no{L^2}{\bdy^\de\bu}\no{L^2}{\nabla\bdy^\be\bu}\no{L^\infty}{\bdy^\al\bu}\notag\\
		&\leq C\no{L^2}{\De\bdy^\al\bu}^{\frac{2(k+1)}{k+2}}\no{L^2}{\bu}^{\frac{k+4}{k+2}}.\notag
	\end{align}
An application of Young's inequality again yields
	\begin{align}
		\sum_{0<|\de|,|\be|<k}c_{\de,\be}|\lb \bdy^\de\bu^j, \bdy_j\bdy^\be\bu^\ell\bdy^\al\bu^\ell\rb|\leq \frac{1}8\no{L^2}{\De\bdy^\al\bu}^2+C\no{L^2}{\bu}^{\frac{k+4}{k+1}}.\notag
	\end{align}
We are left to treat one more term:
	\begin{align}
		\lam|\lb \De\bu,\bdy^{2\al}\bu\rb|&\leq \lam\no{L^2}{\De\bdy^\al\bu}\no{L^2}{\bdy^\al\bu}\notag\\
		&\leq \frac{1}8\no{L^2}{\De\bdy^\al\bu}^2+C\lam^2\no{L^2}{\bdy^\al\bu}^2.\notag
	\end{align}
Finally, combining the above and adding $\no{L^2}{\bu}^2$ to both sides of the resulting inequality, we arrive at
	\begin{align}\notag
		\frac{d}{dt}\no{H^k}{\bu}^2+\no{H^{k+2}}{\bu}^2\leq C\lam^2\no{H^k}{\bu}^2+C(1+\no{L^2}{\bu}^2)^{\frac{k+4}{2(k+1)}}.
	\end{align}
By \eqref{est:L2} and an application of Gr\"onwall's inequality, it follows that 
	\begin{align}\notag
		\sup_{t\in[0,T]}\left(\no{H^k}{\bu(t)}^2+\int_0^t\no{H^{k+2}}{\bu(s)}^2\ ds\right)\leq e^{C\lam^2T}C_*(T,\no{L^2}{\bu_0})^{\frac{k+4}{k+1}},
	\end{align}
for all $T<T^*$, where $C_*$ is the same constant from \eqref{est:L2:local}, and $T^*$ is the same existence time appearing there. 

Finally, we see that we may bootstrap from the fact that $\bu(t)\in H^2$ a.e. $t\in(0,T)$. For such $t_0\in(0,T)$, we may then show $\bu\in C_w([t_0,T];H^2)$, which implies $\bu(t)\in H^2$ for all $t\in [t_0,T)$. Since this holds for a sequence of $t_0\goesto0$, it follows that $\bu(t)\in H^2$, for all $t\in(0,T)$. The above estimates then imply that $\bu\in L^2(0,T;H^{2k})$, where $k=2$. We may now induct on $k$ to deduce $\bu(t)\in H^k$, for all $t\in(0,T)$, for all $k\geq1$.
\end{proof}



\vfill
\begin{minipage}[t]{0.5\textwidth}
\noindent Adam Larios\\
{\footnotesize
Department of Mathematics\\
University of Nebraska--Lincoln \\
\url{https://www.math.unl.edu/~alarios2}\\
\url{alarios@unl.edu}
}
\end{minipage}
\begin{minipage}[t]{0.5\textwidth}
\noindent Vincent R. Martinez$^\dagger$\\
{\footnotesize
Department of Mathematics \& Statistics\\
CUNY Hunter College \\
Department of Mathematics \\
CUNY Graduate Center \\
\url{http://math.hunter.cuny.edu/vmartine/}\\
\url{vrmartinez@hunter.cuny.edu}\\
$\dagger$ corresponding author
}
\end{minipage}

\end{document}